\documentclass[12pt,twoside]{article}
\usepackage{amsmath,amssymb}
\usepackage{amsthm}
\usepackage{upref}
\usepackage{graphicx,amssymb}
\usepackage{color}
\numberwithin{equation}{section}

\theoremstyle{plain}
\newtheorem{theorem}{Theorem}[section]
\newtheorem{proposition}[theorem]{Proposition}

\newtheorem{corollary}[theorem]{Corollary}

\theoremstyle{definition}
\newtheorem{remark}[theorem]{Remark}

\begin{document}

\newcommand{\eq}{equation}
\newcommand{\real}{\ensuremath{\mathbb R}}
\newcommand{\comp}{\ensuremath{\mathbb C}}
\newcommand{\rn}{\ensuremath{{\mathbb R}^n}}
\newcommand{\tn}{\ensuremath{{\mathbb T}^n}}
\newcommand{\rnp}{\ensuremath{\real^n_+}}
\newcommand{\rnn}{\ensuremath{\real^n_-}}
\newcommand{\Rn}{\ensuremath{{\mathbb R}^{n-1}}}
\newcommand{\Zn}{\ensuremath{{\mathbb Z}^{n-1}}}
\newcommand{\no}{\ensuremath{\nat_0}}
\newcommand{\ganz}{\ensuremath{\mathbb Z}}
\newcommand{\zn}{\ensuremath{{\mathbb Z}^n}}
\newcommand{\zom}{\ensuremath{{\mathbb Z}_{\Om}}}
\newcommand{\zOm}{\ensuremath{{\mathbb Z}^{\Om}}}
\newcommand{\As}{\ensuremath{A^s_{p,q}}}
\newcommand{\Bs}{\ensuremath{B^s_{p,q}}}
\newcommand{\Fs}{\ensuremath{F^s_{p,q}}}
\newcommand{\Fsr}{\ensuremath{F^{s,\rloc}_{p,q}}}
\newcommand{\nat}{\ensuremath{\mathbb N}}
\newcommand{\Om}{\ensuremath{\Omega}}
\newcommand{\di}{\ensuremath{{\mathrm d}}}
\newcommand{\sn}{\ensuremath{{\mathbb S}^{n-1}}}
\newcommand{\Ac}{\ensuremath{\mathcal A}}
\newcommand{\Acs}{\ensuremath{\Ac^s_{p,q}}}
\newcommand{\Bc}{\ensuremath{\mathcal B}}
\newcommand{\Cc}{\ensuremath{\mathcal C}}
\newcommand{\Ec}{\ensuremath{\mathcal E}}
\newcommand{\Nc}{\ensuremath{\mathcal N}}
\newcommand{\cc}{{\scriptsize $\Cc$}${}^s (\rn)$}
\newcommand{\ccd}{{\scriptsize $\Cc$}${}^s (\rn, \delta)$}
\newcommand{\Fc}{\ensuremath{\mathcal F}}
\newcommand{\Lc}{\ensuremath{\mathcal L}}
\newcommand{\Mc}{\ensuremath{\mathcal M}}
\newcommand{\Pc}{\ensuremath{\mathcal P}}
\newcommand{\Efr}{\ensuremath{\mathfrak E}}
\newcommand{\Mfr}{\ensuremath{\mathfrak M}}
\newcommand{\Mbf}{\ensuremath{\mathbf M}}
\newcommand{\Dbb}{\ensuremath{\mathbb D}}
\newcommand{\Lbb}{\ensuremath{\mathbb L}}
\newcommand{\Pbb}{\ensuremath{\mathbb P}}
\newcommand{\Qbb}{\ensuremath{\mathbb Q}}
\newcommand{\Rbb}{\ensuremath{\mathbb R}}
\newcommand{\vp}{\ensuremath{\varphi}}
\newcommand{\hra}{\ensuremath{\hookrightarrow}}
\newcommand{\supp}{\ensuremath{\mathrm{supp \,}}}
\newcommand{\ssupp}{\ensuremath{\mathrm{sing \ supp\,}}}
\newcommand{\dist}{\ensuremath{\mathrm{dist \,}}}
\newcommand{\unif}{\ensuremath{\mathrm{unif}}}
\newcommand{\ve}{\ensuremath{\varepsilon}}
\newcommand{\vk}{\ensuremath{\varkappa}}
\newcommand{\vr}{\ensuremath{\varrho}}
\newcommand{\pa}{\ensuremath{\partial}}
\newcommand{\oa}{\ensuremath{\overline{a}}}
\newcommand{\ob}{\ensuremath{\overline{b}}}
\newcommand{\of}{\ensuremath{\overline{f}}}
\newcommand{\LA}{\ensuremath{L^r\!\As}}
\newcommand{\LcA}{\ensuremath{\Lc^{r}\!A^s_{p,q}}}
\newcommand{\LcdA}{\ensuremath{\Lc^{r}\!A^{s+d}_{p,q}}}
\newcommand{\LcB}{\ensuremath{\Lc^{r}\!B^s_{p,q}}}
\newcommand{\LcF}{\ensuremath{\Lc^{r}\!F^s_{p,q}}}
\newcommand{\Lf}{\ensuremath{L^r\!f^s_{p,q}}}
\newcommand{\La}{\ensuremath{\Lambda}}
\newcommand{\Lob}{\ensuremath{L^r \ob{}^s_{p,q}}}
\newcommand{\Lof}{\ensuremath{L^r \of{}^s_{p,q}}}
\newcommand{\Loa}{\ensuremath{L^r\, \oa{}^s_{p,q}}}
\newcommand{\Lcoa}{\ensuremath{\Lc^{r}\oa{}^s_{p,q}}}
\newcommand{\Lcob}{\ensuremath{\Lc^{r}\ob{}^s_{p,q}}}
\newcommand{\Lcof}{\ensuremath{\Lc^{r}\of{}^s_{p,q}}}
\newcommand{\Lca}{\ensuremath{\Lc^{r}\!a^s_{p,q}}}
\newcommand{\Lcb}{\ensuremath{\Lc^{r}\!b^s_{p,q}}}
\newcommand{\Lcf}{\ensuremath{\Lc^{r}\!f^s_{p,q}}}
\newcommand{\id}{\ensuremath{\mathrm{id}}}
\newcommand{\range}{\ensuremath{\mathrm{range \,}}}
\newcommand{\rank}{\ensuremath{\mathrm{rank \,}}}
\newcommand{\tr}{\ensuremath{\mathrm{tr\,}}}
\newcommand{\trd}{\ensuremath{\mathrm{tr}_d}}
\newcommand{\trL}{\ensuremath{\mathrm{tr}_L}}
\newcommand{\ext}{\ensuremath{\mathrm{ext}}}
\newcommand{\re}{\ensuremath{\mathrm{re\,}}}
\newcommand{\Rea}{\ensuremath{\mathrm{Re\,}}}
\newcommand{\Ima}{\ensuremath{\mathrm{Im\,}}}
\newcommand{\loc}{\ensuremath{\mathrm{loc}}}
\newcommand{\rloc}{\ensuremath{\mathrm{rloc}}}
\newcommand{\per}{\ensuremath{\mathrm{per}}}
\newcommand{\osc}{\ensuremath{\mathrm{osc}}}
\newcommand{\pr}{\pageref}
\newcommand{\wh}{\ensuremath{\widehat}}
\newcommand{\wt}{\ensuremath{\widetilde}}
\newcommand{\ol}{\ensuremath{\overline}}
\newcommand{\os}{\ensuremath{\overset}}
\newcommand{\Li}{\ensuremath{\overset{\circ}{L}}}
\newcommand{\Ai}{\ensuremath{\os{\, \circ}{A}}}
\newcommand{\Ci}{\ensuremath{\os{\circ}{\Cc}}}
\newcommand{\Si}{\ensuremath{\overset{\circ}{S}}}
\newcommand{\dom}{\ensuremath{\mathrm{dom \,}}}
\newcommand{\SA}{\ensuremath{S^r_{p,q} A}}
\newcommand{\SB}{\ensuremath{S^r_{p,q} B}}
\newcommand{\SF}{\ensuremath{S^r_{p,q} F}}
\newcommand{\Hc}{\ensuremath{\mathcal H}}
\newcommand{\Lci}{\ensuremath{\overset{\circ}{\Lc}}}
\newcommand{\bmo}{\ensuremath{\mathrm{bmo}}}
\newcommand{\BMO}{\ensuremath{\mathrm{BMO}}}
\newcommand{\cm}{\\[0.1cm]}
\newcommand{\Aa}{\ensuremath{\os{\, \ast}{A}}}
\newcommand{\Ba}{\ensuremath{\os{\, \ast}{B}}}
\newcommand{\Fa}{\ensuremath{\os{\, \ast}{F}}}
\newcommand{\Ha}{\ensuremath{\os{\, \ast}{H}}}
\newcommand{\Aas}{\ensuremath{\Aa{}^s_{p,q}}}
\newcommand{\Bas}{\ensuremath{\Ba{}^s_{p,q}}}
\newcommand{\Fas}{\ensuremath{\Fa{}^s_{p,q}}}
\newcommand{\Ca}{\ensuremath{\os{\, \ast}{{\mathcal C}}}}
\newcommand{\Cas}{\ensuremath{\Ca{}^s}}
\newcommand{\Car}{\ensuremath{\Ca{}^r}}
\newcommand{\bl}{$\blacksquare$}

\begin{center}
{\Large Poisson summation formulas: A new approach}
\\[1cm]
{Hans Triebel}
\\[0.2cm]
Institut f\"{u}r Mathematik\\
Friedrich--Schiller--Universit\"{a}t Jena\\
07737 Jena, Germany
\\[0.1cm]
email: hans.triebel@uni-jena.de
\end{center}

\begin{abstract}
The paper deals with a new approach to Poisson summation formulas in the context of function spaces on \rn.
\end{abstract}



\section{Introduction}   \label{S1}
Let $\vp \in S(\rn)$. Then
\begin{\eq}  \label{1.1}
\sum_{k\in \zn} \vp(2\pi k) = (2\pi)^{-n/2} \sum_{k \in \zn} \wh{\vp} (k)
\end{\eq}
is the classical Poisson summation formula, where as usual
\begin{\eq}   \label{1.2}
\wh{\vp} (\xi) = (2\pi )^{-n/2} \int_{\rn} e^{-ix \xi} \vp(x) \, \di x, \qquad \xi \in \rn,
\end{\eq}
is the Fourier transform of \vp. This relation has been modified and generalized in many directions. There are numerous applications
in analytic number theory and physics, in particular phase--frequency assertions for crystals and quasi--crystals, up to our time. We 
refer the reader to \cite{LeO15, LeO17, Mey16, Mey17}, the recent papers \cite{Gon23, GoV25} and the references within. We offer a new 
approach, based on the observation
\begin{\eq}   \label{1.3}
\sum_{k \in \zn}  e^{i2\pi kx} = \sum_{k\in \zn} \delta_k \in \Cc^{-n} (\rn) \hra S'(\rn).
\end{\eq}
Here $\delta_k$ are the classical Dirac distributions, $\delta_k (\vp) = \vp (k)$, $k \in \zn$, $\vp \in S(\rn)$, whereas $e^{i2\pi kx}
\in S'(\rn)$ must be considered as related building blocks in the framework of the dual pairing $\big( S(\rn), S'(\rn) \big)$ with
$kx = \sum^n_{j=1} k_j x_j$, $k \in \zn$, $x\in \rn$. This identity in $S'(\rn)$ can be inserted in distinguished operators. The 
outcome are generalized Poisson summation formulas. Secondly one can first apply standard procedures in the theory of function spaces
in $\rn$ to \eqref{1.3}, for example diffeomorphisms, to produce more general versions and to proceed afterwards as indicated.

In Section \ref{S2} we justify \eqref{1.3}. This will be applied and modified in Section \ref{S3} as outlined above. 

\section{Preliminaries and main assertions}   \label{S2}
\subsection{Preliminaries}   \label{S2.1}
First we collect some notation. Let $\nat$ be the collection of all natural numbers and $\no = \nat \cup \{0 \}$. Let $\rn$ be Euclidean 
$n$--space, where $n \in \nat$. Put $\real = \real^1$ whereas $\comp$ is the complex plane. As usual $\ganz$ is the collection of all
integers, and $\zn$, where $n\in \nat$, denotes the lattice of all points $m = (m_1, \ldots, m_n) \in \rn$ with $m_j \in \ganz$, $j=
1, \ldots,n$. Let $\nat^n_0$, where $n\in \nat$, be the set of all multi--indices, $\alpha = (\alpha_1, \ldots, \alpha_n)$ with
$\alpha_j \in \no$ and $|\alpha| = \sum^n_{j=1} \alpha_j$. As usual, derivatives are abbreviated by
\begin{\eq}    \label{2.1}
D^\alpha = \frac{\pa^{|\alpha|}}{\pa x^{\alpha_1}_1 \ldots \pa x^{\alpha_n}_n}, \qquad \alpha \in \nat^n_0, \quad x \in \rn,
\end{\eq}
and similarly
\begin{\eq}   \label{2.2}
\xi^\alpha = \xi^{\alpha_1}_1 \cdots \xi^{\alpha_n}_n, \qquad \alpha \in \nat^n_0, \quad \xi \in \rn.
\end{\eq}
For $x\in \rn$, $y\in \rn$, let
\begin{\eq}   \label{2.3}
xy = \sum^n_{j=1} x_j y_j, \qquad x= (x_1, \ldots, x_n), \quad y = (y_1, \ldots, y_n),
\end{\eq}
and $|x| = \sqrt{xx}$.

Let $S(\rn)$ be the usual Schwartz space of all rapidly decreasing, infinitely differentiable functions on $\rn$ and let $S'(\rn)$ be its dual, the space of tempered distributions. Let $\vp \in S(\rn)$. Then
\begin{\eq}   \label{2.4}
\wh{\vp} (\xi) = \big(F \vp \big)(\xi) = (2\pi)^{-n/2} \int_{\rn}  e^{-ix \xi} \, \vp (x) \, \di x, \qquad \xi \in \rn,
\end{\eq}
denotes the Fourier transform of \vp, and
\begin{\eq}   \label{2.5}
\vp^\vee (\xi) = \big( F^{-1} \vp \big)(\xi) =(2 \pi)^{-n/2} \int_{\rn} e^{i x \xi} \vp (x) \, \di x,\qquad \xi \in \rn,
\end{\eq}
the inverse Fourier transform of \vp. This is extended in the standard way to $S'(\rn)$ by
\begin{\eq}   \label{2.5a}
(FT)(\vp) = \wh{T} (\vp) = T(\wh{\vp} ) \quad \text{and} \quad (F^{-1}T )(\vp) =  T^\vee (\vp) = T (\vp^\vee)
\end{\eq}
for $\vp \in S(\rn)$ and $T\in S'(\rn)$.

We need the periodic counterpart of the above Fourier transform in the $n$--torus $\tn$,
\begin{\eq}   \label{2.5b}
\tn = \{ x= (x_1, \ldots, x_n) \in \rn: \ -\pi \le x_j \le \pi, \ j =1, \ldots, n \}, \quad n \in \nat,
\end{\eq}
identified with the indicated cube in \rn. Then
\begin{\eq}   \label{2.6}
f(x) = (2 \pi)^{-n/2}   \sum_{m \in \zn} a_m e^{imx}, \qquad x\in \tn,
\end{\eq}
with
\begin{\eq}    \label{2.7}
a_m = (2\pi)^{-n/2} \int_{\tn} e^{-imx} f(x) \, \di x, \qquad m\in \zn,
\end{\eq}
is the classical representation of periodic functions  in \tn. 
The related theory of periodic functions $f \in L_1 (\tn)$ can be
extended to the theory of periodic distributions based
on the dual pairing $\big( D(\tn), D' (\tn) \big)$,
\begin{\eq}   \label{2.8}
(f, \vp)_\pi, \qquad f \in D'(\tn), \quad \vp \in D(\tn).
\end{\eq}
This will not be repeated here. We refer the reader to \cite[Chapter 3]{ST87} and \cite[Section 1.3]{T08}. On this basis one can
extend \eqref{2.6} and \eqref{2.7} to
\begin{\eq}   \label{2.9}
f = (2 \pi)^{-n/2}   \sum_{m \in \zn} a_m e^{imx}, \qquad f \in D'( \tn),
\end{\eq}
convergence in $D'(\tn)$, with the Fourier coefficients
\begin{\eq}     \label{2.10}
a_m = (2\pi)^{-n/2} \big(f, e^{-im x} \big)_\pi, \qquad m \in \zn, \quad f \in D'(\tn).
\end{\eq}
One can extend $f\in D'(\tn)$ periodically to \rn,
\begin{\eq}   \label{2.11}
f(\cdot) = f(\cdot + 2\pi k), \qquad k \in \zn,
\end{\eq}
to $S_\pi '(\rn)$, the periodic subspace of $S'(\rn)$. Then one has now
\begin{\eq}    \label{2.12}
f = (2 \pi)^{-n/2}   \sum_{m \in \zn} a_m e^{imx}, \qquad f \in S'_\pi (\rn),
\end{\eq}
instead of \eqref{2.9} with
\begin{\eq}     \label{2.13}
a_m = (2\pi)^{-n/2} \big(f, e^{-im x} \big)_\pi, \qquad m \in \zn.
\end{\eq}
Recall that 
\begin{\eq} \label{2.14}
F \big( e^{imx} \big) = (2 \pi)^{n/2} \delta_m, \quad \text{and} \quad F\delta_m =(2\pi)^{-n/2} e^{-imx}
\end{\eq}
where $\delta_m$ is the $\delta$--distribution  with  $m \in \zn$ as off--point, $\delta_m \vp = \vp(m)$, $\vp \in S(\rn)$. Applied to 
\eqref{2.12} one obtains
\begin{\eq}   \label{2.15}
Ff = \sum_{m\in \zn} a_m \delta_m, \qquad a_m = (2 \pi)^{-n/2} \big( f, e^{-imx} \big)_\pi.
\end{\eq}

\subsection{Main assertions}    \label{S2.2}
First we justify an extended version of the above classical Poisson summation formula \eqref{1.1}.

Let $\vp \in S(\rn)$ and
\begin{\eq}    \label{2.16}
f(x) = \sum_{k\in \zn} \vp(x+ 2\pi k), \qquad x\in \rn.
\end{\eq}
Then $f$ and all its derivatives $D^\alpha f$, $\alpha \in \nat^n_0$, are periodic bounded $C^\infty$ functions in \rn. One has by 
\eqref{2.12}, \eqref{2.13} that
\begin{\eq}   \label{2.17}
f(x) = (2\pi)^{-n/2} \sum_{m \in \zn} \wh{f}_\pi (m) \, e^{imx}
\end{\eq}
with
\begin{\eq}   \label{2.18}
\begin{aligned}
\wh{f}_\pi (m) &= (2\pi)^{-n/2} \big( f, e^{-imx} \big)_\pi \\
&= (2\pi)^{-n/2} \sum_{k \in \zn} \int_{\tn} \vp( x + 2\pi k) \, e^{-imx} \, \di x \\
&= (2\pi)^{-n/2} \sum_{k \in \zn} \int_{\tn + 2\pi k} \vp(x) \, e^{-imx} \, \di x \\
&= \wh{\vp} (m).
\end{aligned}
\end{\eq}
Inserted in \eqref{2.17} one obtains 
\begin{\eq}   \label{2.19}
\sum_{k \in \zn} \vp (x+ 2\pi k) = (2\pi)^{-n/2} \sum_{k \in \zn} \wh{\vp} (k) \, e^{ikx}, \qquad x\in \rn.
\end{\eq}
Then $x=0$ justifies \eqref{1.1}.

After these preparations we can now prove \eqref{1.3} rather quickly. For this purpose we have first to say what is meant by the 
function space $\Cc^s (\rn)$, $s\in \real$. Let $\vp_0 \in S(\rn)$ with
\begin{\eq}   \label{2.20}
\vp_0 (x) =1 \ \text{if $|x| \le 1$} \qquad \text{and} \qquad \vp_0 (x) =0 \ \text{if $|x| \ge 3/2$},
\end{\eq}
and let
\begin{\eq}   \label{2.21}
\vp_k (x) = \vp_0 \big( 2^{-k} x) - \vp_0 \big( 2^{-k+1} x \big), \qquad x \in \rn, \quad k \in \nat.
\end{\eq}
Since
\begin{\eq}    \label{2.22}
\sum^\infty_{j=0} \vp_j (x) = 1 \qquad \text{for} \quad x\in \rn,
\end{\eq}
$\vp = \{\vp_j\}^\infty_{j=0}$ forms a dyadic resolution of unity. The entire analytic functions $(\vp_j \wh{f} )^\vee (x)$ make sense pointwise in $\rn$ for any $f \in S' (\rn)$. Let $s\in \real$. Then $\Cc^s(\rn) = B^s_{\infty, \infty} (\rn) = F^s_{\infty, \infty} (\rn)$ is the collection of all $f\in S'(\rn)$ such that
\begin{\eq}  \label{2.23}
\| f \, | \Cc^s (\rn) \|_\vp = \sup_{j\in \no, x\in \rn} 2^{js} \big| (\vp_j \wh{f} )^\vee (x) \big|
\end{\eq}
is finite. This is a special case of the function spaces $\As (\rn)$, $A\in \{B,F \}$, $s\in \real$ and $0<p,q \le \infty$. One may
consult \cite[Section 1.1.1, pp.~1--5]{T20}. There one finds explanations, discussions and references. It is well known that these
Banach spaces are independent of the chosen resolution of unity $\vp$. This justifies our omission of the subscript $\vp$ in 
\eqref{2.23} in the sequel. For later applications it might be useful to mention that the general assertion
\begin{\eq}   \label{2.24}
\Cc^s (\rn) \hra S'(\rn), \qquad s\in \real,
\end{\eq}
(continuous embedding) can be complemented as follows. If $0<s<1$, then $\Cc^s (\rn) = C^s (\rn)$ are the classical H\"{o}lder spaces,
whereas $\Cc^1 (\rn)$ is the Zygmund class. Furthermore,
\begin{\eq}  \label{2.25}
\Cc^s (\rn) \subset L^{\loc}_1 (\rn), \quad \text{regular distributions, if, and only if, $s>0$},
\end{\eq}
\begin{\eq}   \label{2.26}
\Cc^s (\rn) \hra L_\infty (\rn) \quad \text{if, and  only if, $s>0$},
\end{\eq}
\begin{\eq}   \label{2.27}
\Cc^s (\rn) \hra C(\rn) \quad \text{if, and only if, $s>0$},
\end{\eq}
where $C(\rn)$ is the space of all bounded continuous functions in \rn. These assertions are special cases of 
\cite[Theorems 2.3, 2.4, pp.~22--24]{T20}. Otherwise the convergence of the series below must be understood in the framework of the
dual pairing $\big(S(\rn), S'(\rn) \big)$. Recall that $\delta_a$ is the usual Dirac distribution with the off--point $a\in \rn$, $
\delta_a (\vp) = \vp(a)$, $\vp \in S(\rn)$. Furthermore $e^{iax} \in S'(\rn)$ for fixed $a= (a_1, \ldots, a_n) \in \rn$ must be 
considered as a distribution, say, $e_a$, generated by $e_a (x) = e^{iax}$, $x= (x_1, \ldots, x_n) \in \rn$, with $ax = \sum_{j=1}^n
a_j x_j$. But we stick at the notation $e^{iax}$. In extension of \eqref{2.14} we recall that (in the above interpretation)
\begin{\eq}   \label{2.28}
F (e^{iax}) = (2\pi)^{n/2} \delta_a \quad \text{and} \quad F\delta_a = (2\pi)^{-n/2} e^{-iax}
\end{\eq}
for $a= (a_1, \ldots, a_n) \in \rn$ as fixed off--point. Inserted in \eqref{2.23} it follows that $\delta_a \in \Cc^s (\rn)$ if, and
only if, $s \le -n$. 

\begin{theorem}   \label{T2.1}
Let $n \in \nat$. Then
\begin{\eq}   \label{2.29}
\sum_{k\in \zn} e^{i2 \pi kx} = \sum_{k\in \zn} \delta_k \in \Cc^{-n} (\rn),
\end{\eq}
convergence being in $S'(\rn)$.
\end{theorem}

\begin{proof} {\em Step 1.} Let $\vp \in S(\rn)$. We apply the Fourier transform to \eqref{2.19}. Using $\big( F\vp(\cdot +h) \big)(x)
= e^{ihx} F\vp (x)$ and \eqref{2.28} one obtains
\begin{\eq}   \label{2.30}
\wh{\vp} (x) \sum_{k \in \zn} e^{i2 \pi kx} = \sum_{k \in \zn} \wh{\vp}(k) \, \delta_k.
\end{\eq}
\cm
{\em Step 2.} Next we check that both sides of \eqref{2.29} belong to $\Cc^{-n} (\rn)$. As far as the right--hand side is concerned we
remark that
\begin{\eq}   \label{2.31}
\Big\| \sum_{k \in \zn} a_k \delta_k \, | \Cc^s (\rn) \Big\| \sim \sup_{k \in \zn} \| a_k \delta_k \, | \Cc^s (\rn) \| \sim \sup_{k \in
\zn} |a_k|
\end{\eq}
for any $s \le -n$ and any bounded set $\{a_k \} \subset \comp$ is a consequence of $\delta_k \in \Cc^s(\rn)$and the location property
according to \cite[Section 2.4.2, pp.~43--44]{T20}, combined with the Fatou property as described in \cite[Section 1.3.4, 
p.~18]{T20} (ensuring that one can step from finite partial sums  in \eqref{2.31} to the total sum). We wish to justify (independently)
that also the left--hand side of \eqref{2.29} belongs to $\Cc^{-n} (\rn)$. For this purpose we insert $f= \sum_{k\in \zn} e^{i2\pi kx}$
in \eqref{2.23} with $s \le-n$. One has by \eqref{2.28} that
\begin{\eq}   \label{2.32}
\big(\vp_j \wh{f} \big)^\vee (x) = \sum_{k \in \zn} \vp_j (2\pi k) e^{i2\pi kx}, \qquad x\in \rn,
\end{\eq}
and for some $c>0$,
\begin{\eq}   \label{2.33}
\big| \big( \vp_j \wh{f} \big)^\vee (x) \big| \le c \, 2^{jn}, \qquad j \in \no.
\end{\eq}
Then it follows from \eqref{2.23} that the left--hand side of \eqref{2.29} is a converging series in $\Cc^s (\rn)$ with $s<-n$. In 
particular it is an element of $S'(\rn)$. Then \eqref{2.33} inserted in \eqref{2.23} with $s \le -n$ shows that the left-hand side of
\eqref{2.29} belongs to $\Cc^{-n} (\rn)$. 
\cm
{\em Step 3.} Both sides of \eqref{2.29} are elements of $S'(\rn)$. It follows from Step 1 that they coincide locally. Then they 
coincide globally. This proves the theorem.
\end{proof}

\section{Applications}   \label{S3}
\subsection{Lifts}    \label{S3.1}
Let $n\in \nat$ and
\begin{\eq}   \label{3.1}
w_\alpha (x) = (1 + |x|^2)^{\alpha/2}, \qquad \alpha \in \real, \quad x\in \rn.
\end{\eq}
Then $I_\alpha$,
\begin{\eq}   \label{3.2}
I_\alpha: \quad f \mapsto (w_\alpha \wh{f} )^\vee = (w_\alpha f^\vee )^\wedge, \qquad f\in S'(\rn), \quad \alpha \in \real,
\end{\eq}
maps $S(\rn)$ one--to--one onto itself and also $S'(\rn)$ one--to--one onto itself. Furthermore, $I_\alpha$ maps any function space
\begin{\eq}   \label{3.3}
\As (\rn), \quad A\in \{B,F \}, \quad s\in \real \quad \text{and} \quad 0<p,q \le \infty,
\end{\eq}
isomorphically onto $A^{s-\alpha}_{p,q} (\rn)$,
\begin{\eq}   \label{3.4}
\begin{aligned}
I_\alpha \As (\rn) &= A^{s-\alpha}_{p,q} (\rn), \\
\| (w_\alpha \wh{f})^\vee | A^{s-\alpha}_{p,q} (\rn) \| &\sim \|f \, | \As (\rn) \|,
\end{aligned}
\end{\eq}
including in particular $\Cc^s(\rn) = B^s_{\infty, \infty} (\rn) = F^s_{\infty, \infty} (\rn)$. One may consult \cite[Section 1.3.2,
pp.~16--17]{T20} and the references given there. We will not need  these lifting properties explicitly. But we ask what happens if
one inserts both sides of \eqref{2.29} into $I_\alpha$. One needs $\wh{w_\alpha}$. If $\alpha <0$ then
\begin{\eq}   \label{3.5}
\wh{w_\alpha} (x) = \frac{(4\pi )^{\alpha/2} (2\pi )^{n/2}}{\Gamma ({\frac{|\alpha|}{2})}} \int^\infty_0 t^{-\frac{\alpha +n}{2}}
e^{- \frac{\pi |x|^2}{t} - \frac{t}{4 \pi}} \, \frac{\di t}{t}, \qquad x\in \rn \setminus \{0\},
\end{\eq}
according to \cite[Section 1.2.4, pp.\,10--11]{AdH96} with a reference to \cite[V3.1, pp.~130--133]{Ste70}. A detailed discussion of
these so--called Bessel potentials $\wh{w_\alpha} (x)$ with $-n < \alpha <0$ near the origin $x=0$ and at infinity may be found in
\cite[Sections 1.2.3--1.2.5, pp.~9--13]{AdH96}. The more handsome formula for $\wh{w_{1-n}} (x)$, $2\le n \in \nat$, $x\in \rn \setminus \{0\}$, will be mentioned below indicating in particular the behaviour  near the origin and at infinity. 

\begin{proposition} \label{P3.1}
Let $n\in \nat$ and $\alpha \in \real$.
\cm
{\em (i)} Then
\begin{\eq}   \label{3.6}
\sum_{k \in \zn} \big(1 + |2\pi k|^2)^{\alpha/2}\, e^{i2\pi kx} = (2\pi)^{-n/2} \sum_{k\in \zn} \wh{w_\alpha} (x-k),
\end{\eq}
convergence being in $S'(\rn)$.
\cm
{\em (ii)} If $\alpha <-n$ then the above functions are bounded and continuous in \rn. If $n \ge 2$ and $\alpha =1-n$ then
\begin{\eq}   \label{3.7}
\sum_{k \in \zn} \frac{e^{i2\pi kx}}{ (1+ |2\pi k|^2)^{\frac{n-1}{2}}} = \frac{1}{(2\sqrt{\pi})^{n-1} \Gamma (\frac{n-1}{2})} \sum_{k\in \zn} \frac{e^{-|x-k|}}{|x-k|},
\end{\eq}
$($regular distributions belonging to $S'(\rn) \cap L^{\loc}_1 (\rn))$.
\end{proposition}

\begin{proof} {\em Step 1.} We apply \eqref{2.29}. One has by \eqref{3.2} and \eqref{2.28} that
\begin{\eq}   \label{3.8}
\begin{aligned}
I_\alpha \Big( \sum_{k \in \zn} e^{i2\pi kx} \Big) &= (2\pi)^{n/2}  \sum_{k \in \zn} \big( (1+ |\xi|^2)^{\alpha/2} \delta_{2\pi k}
\big)^\vee \\
&= \sum_{k \in \zn} \big( 1 + |2\pi k|^2 \big)^{\alpha/2} \, e^{i2\pi kx}.
\end{aligned}
\end{\eq}
Furthermore, 
\begin{\eq}   \label{3.9}
\begin{aligned}
I_\alpha \Big( \sum_{k\in \zn} \delta_k \Big) &= \sum_{k\in \zn} \big( (1+|\xi|^2)^{\alpha/2} \wh{\delta_k} \big)^\vee \\
&= (2\pi)^{-n/2} \sum_{k\in \zn} \big( (1+ |\xi|^2)^{\alpha/2} e^{-ik \xi} \big)^\vee \\
&= (2\pi)^{-n/2} \sum_{k \in \zn} \wh{w_\alpha} (x-k).
\end{aligned}
\end{\eq}
Applied to \eqref{2.29} one obtains \eqref{3.6}.
\cm
{\em Step 2.} If $\alpha <-n$ the the sums in \eqref{3.6} converge pointwise to a bounded continuous function. 
If $\alpha =1-n$, $n \ge 2$, then \eqref{3.5} can be reduced to more elementary functions. Adapting  the corresponding
calculations in \cite[Sections 1.2.4, 1.2.5, pp.~10--13]{AdH96} one obtains the  indicated  regular distribution  in \eqref{3.7}.
\end{proof}

\subsection{Gauss--Weierstrass semi--groups}    \label{S3.2}
The Gauss--Weierstrass semi--group $W_t$,
\begin{\eq}   \label{3.10}
\begin{aligned}
W_t f(x) &= \frac{1}{(4\pi t)^{n/2}} \int_{\rn} e^{- \frac{|x-y|^2}{4t}} f(y) \, \di y 
= \frac{1}{(4\pi t)^{n/2}} \Big(f, e^{- \frac{|x- \cdot|^2}{4t}} \Big) \\
&=  \big( e^{-t|\cdot|^2} \wh{f} (\cdot) \big)^\vee (x), \qquad t>0,\quad x\in \rn,
\end{aligned}
\end{\eq}
plays a fundamental role in the theory of function spaces and also other parts of mathematics. Further information may be found in
\cite[Section 1.2, pp.~4--5]{T15}, \cite[p.~106]{T20} and the references given there. This may justify to insert both sides of
\eqref{2.29} into the last version of \eqref{3.10} which makes sense for any $f\in S'(\rn)$.

\begin{proposition}   \label{P3.2}
Let $n \in \nat$ and $t>0$. Then 
\begin{\eq}    \label{3.11}
\sum_{k\in \zn} e^{-t \pi^2 |k|^2} e^{i2 \pi kx} = \frac{1}{(t \pi)^{n/2}} \sum_{k \in \zn} e^{-\frac{|x-k|^2}{t}}, \qquad x\in \rn.
\end{\eq}
\end{proposition}

\begin{proof}
One has by \eqref{3.10} and \eqref{2.28} that
\begin{\eq}   \label{3.12}
\begin{aligned}
W_t \Big( \sum_{k\in \zn} e^{i2\pi k \xi} \Big)(x) &= (2\pi)^{n/2} \sum_{k \in \zn} \big( e^{-t|\cdot|^2} \delta_{2\pi k} \big)^\vee
(x) \\
&= \sum_{k\in \zn} e^{-t|2\pi k|^2} e^{i2\pi kx},
\end{aligned}
\end{\eq}
first in the context of the dual pairing $\big(S(\rn), S'(\rn) \big)$, then pointwise $x\in \rn$. Recall that
\begin{\eq}   \label{3.13}
\big( e^{- \frac{|\cdot|^2}{2a}} \big)^\wedge (x) = a^{\frac{n}{2}} e^{-\frac{a}{2} |x|^2}, \qquad x\in \rn, \quad a>0.
\end{\eq}
Then it follows from \eqref{2.28} and \eqref{3.13} with $2at =1$ that
\begin{\eq}   \label{3.14}
\begin{aligned}
W_t \Big( \sum_{k \in \zn} \delta_k \Big)(x) &= (2\pi)^{-n/2} \sum_{k\in \zn} \big( e^{-t|\xi|^2} e^{-ik \xi} \big)^\vee (x) \\
&= (2\pi)^{-n/2} \sum_{k\in \zn} \big( e^{-t |\xi|^2} \big)^\vee (x-k) \\
&= (2\pi)^{-n/2} \sum_{k \in \zn} \big( \frac{1}{2t} )^{n/2} e^{-\frac{|x-k|^2}{4t}}
\end{aligned}
\end{\eq}
Then \eqref{3.11} follows from \eqref{2.29} and \eqref{3.12}, \eqref{3.14} with $t$ in place of $4t$.
\end{proof}

\begin{remark}   \label{R3.3}
By \eqref{3.11} one has
\begin{\eq}  \label{3.15}
\sum_{k\in \zn} e^{-t|k|^2} = \big(\frac{\pi}{t} \big)^{\frac{n}{2}} \sum_{k\in \zn} e^{- \frac{\pi^2 |k|^2}{t}}, \qquad t>0.
\end{\eq}
The case $n=1$ is a famous transformation formula for one of the so--called theta functions, \cite[10.11, p.~182]{Edw79} playing a
role in analytic number theory.
\end{remark}

\subsection{Cauchy--Poisson semi--groups}   \label{S3.3}
Next we ask what happens if one inserts both sides of \eqref{2.29} into the Cauchy--Poisson semi--group $P_t$,
\begin{\eq}   \label{3.16}
\begin{aligned}
P_t f(x) &= \big( e^{-t|\cdot|} \wh{f}(\cdot) \big)^\vee (x) \\
&= c_n \int_{\rn} \frac{t}{\big( |x-y|^2 + t^2 \big)^{\frac{n+1}{2}}} \, f(y) \, \di y, \qquad x\in \rn, \quad t>0,
\end{aligned}
\end{\eq}
with
\begin{\eq}  \label{3.17}
c_n  \int_{\rn} (1+|x|^2)^{-\frac{n+1}{2}} \, \di x  =1.
\end{\eq}
Similarly as the Gauss--Weierstrass semi--group $W_t$, treated in Section \ref{S3.2} above, also $P_t$ plays a substantial role in
the theory of function spaces. The classical part, including \eqref{3.17} and
\begin{\eq}   \label{3.18}
\big( e^{-t |\xi|} \big)^\vee (x) = (2\pi)^{n/2} c_n \frac{t}{(|x|^2 + t^2)^{\frac{n+1}{2}}}, \qquad t>0, \quad x\in \rn,
\end{\eq}
may be found in \cite[Section 2.5.3, pp.~192--196]{T78}. For more recent assertions and related references we refer the reader to
\cite{T15}. But this will not be needed here. The counterpart of Proposition \ref{P3.2} now based on $P_t$ instead of $W_t$ reads
as follows.

\begin{proposition}   \label{P3.4}
Let $n \in \nat$ and $t>0$. Then
\begin{\eq}   \label{3.19}
\sum_{k\in \zn}  e^{-2\pi t |k|} e^{i2 \pi kx} = c_n t  \sum_{k\in \zn} \frac{1}{(|x-k|^2 + t^2)^{\frac{n+1}{2}}}, \qquad x\in \rn,
\end{\eq}
with $c_n$ as in \eqref{3.17}.
\end{proposition} 

\begin{proof}
Instead of \eqref{3.12} one has now
\begin{\eq}   \label{3.20}
\begin{aligned}
P_t \Big( \sum_{k\in \zn} e^{i2 \pi k\xi} \Big) (x) &= (2\pi)^{n/2} \sum_{k\in \zn} \big( e^{-t|\cdot|} \delta_{2 \pi k} \big)^\vee
(x) \\
&= \sum_{k\in \zn} e^{-t2 \pi |k|} e^{i2\pi kx}.
\end{aligned}
\end{\eq}
In the counterpart of \eqref{3.14} we use in addition \eqref{3.18}. Then one obtains that
\begin{\eq}   \label{3.21}
\begin{aligned}
P_t \Big( \sum_{k\in \zn} \delta_k \Big)(x) &= (2\pi)^{-n/2} \sum_{k\in \zn} \big( e^{-t|\xi|} e^{-ik \xi}\big)^\vee (x) \\
&= c_n \sum_{k\in \zn} \frac{t}{(|x-k|^2 + t^2)^{\frac{n+1}{2}}}
\end{aligned}
\end{\eq}
Then \eqref{3.19} follows from \eqref{2.29} and \eqref{3.20}, \eqref{3.21}.
\end{proof}

Again one may ask for applications. For $x=0$ and $t=1$ one has
\begin{\eq}   \label{3.22}
\sum_{k\in \zn} e^{-2\pi |k|} = c_n \sum_{k\in \zn} \frac{1}{(1+|k|^2)^{\frac{n+1}{2}}}.
\end{\eq}
This can be calculated for $n=1$ as follows.

\begin{corollary}   \label{C3.5}
It holds
\begin{\eq}   \label{3.23}
2 \sum_{k=0}^\infty \frac{1}{1+k^2}=\pi \frac{1+ e^{-2\pi}}{1- e^{-2\pi}} +1.
\end{\eq}
\end{corollary}

\begin{proof}
One has by \eqref{3.22} with $n=1$ that 
\begin{\eq}   \label{3.24}
1 + \sum_{k=1}^\infty \frac{2}{1+ k^2} = \frac{1}{c_1} \Big( \frac{2}{1-e^{-2\pi}} -1 \Big).
\end{\eq}
Together with $c^{-1}_1 = \int_{\real} (1+ x^2)^{-1} \, \di x = \pi$ one obtains \eqref{3.23}.
\end{proof}

\subsection{Fourier operators and pseudodifferential operators}   \label{S3.4}
We call $F_\tau$,
\begin{\eq}   \label{3.25}
(F_\tau f)(x) = \int_{\rn} e^{-ix \xi} \tau (x,\xi) f (\xi) \, \di \xi, \qquad x\in \rn,
\end{\eq}
a {\em Fourier operator} of the class $\Phi^\sigma_{1,\delta} (\rn)$ with $\sigma \in \real$ and $0 \le \delta
< 1$ if the symbol $\tau (x,\xi) \in C^\infty (\real^{2n} )$ satisfies for some constants $c_{\alpha, \gamma} \ge 0$,
\begin{\eq}   \label{3.26}
\big| D^\alpha_x D^\gamma_\xi \tau (x, \xi) \big| \le c_{\alpha, \gamma} (1+ |\xi|)^{\sigma - |\gamma| + \delta |\alpha|}, \qquad
x\in \rn, \quad \xi \in \rn,
\end{\eq}
$\alpha \in \nat^n_0$, $\gamma \in \nat^n_0$. This is the direct counterpart of the H\"{o}rmander
class $\Psi^\sigma_{1, \delta}(\rn)$ of the
{\em pseudodifferential operators} $T_\tau$,
\begin{\eq}   \label{3.27}
(T_\tau f)(x) = \int_{\rn} e^{-ix \xi} \tau (x,\xi) f^\vee (\xi) \, \di \xi, \qquad x\in \rn,
\end{\eq}
with $\tau$ as in \eqref{3.26}. In particular one can compose 
$F_\tau \in \Phi^\sigma_{1,\delta} (\rn)$ as
\begin{\eq}   \label{3.28}
F_\tau = T_\tau \circ F, \qquad T_\tau \in \Psi^\sigma_{1, \delta} (\rn),
\end{\eq}
where $F$ is again the Fourier transform. The well--known mapping properties
\begin{\eq}   \label{3.29}
T_\tau: \quad S(\rn) \hra S(\rn) \quad \text{and} \quad T_\tau: \quad S'(\rn) \hra S'(\rn)
\end{\eq}
are covered by \cite[pp.~68, 70, 94]{Hor85}. Then it follows from \eqref{3.27} and basic mapping properties of $F$ that also
\begin{\eq}   \label{3.30}
F_\tau: \quad S(\rn) \hra S(\rn) \quad \text{and} \quad F_\tau: \quad S'(\rn) \hra S'(\rn).
\end{\eq}
This shows that one can insert the identity \eqref{2.29} both in $T_\tau$ and $F_\tau$. The outcome is more or less the same and it is
sufficient to deal with $F_\tau$. Furthermore we are not interested in most general assertions. We wish to show what type of assertions
can be expected. We deal with the model case $\tau \in \Phi (\rn)$, consisting of all symbols $\tau \in \Phi^0_{1,0} (\rn)$ which are 
independent of $x\in \rn$. This means that $\tau$ is a $C^\infty$ function in $\rn$ such that for some $c_\gamma \ge 0$,
\begin{\eq}   \label{3.31}
|D^\gamma_\xi \tau (\xi) | \le c_\gamma (1+ |\xi|)^{-|\gamma}, \qquad \xi \in \rn, \quad \gamma \in \no.
\end{\eq}

\begin{proposition}   \label{P3.6}
Let $\tau \in \Phi (\rn)$. Then
\begin{\eq}   \label{3.31a}
\sum_{k\in \zn} \tau (k) e^{-ixk} = (2\pi)^{n/2} \sum_{k\in \zn} \wh{\tau} (x- 2\pi k),
\end{\eq}
convergence being in $S'(\rn)$.
\end{proposition}

\begin{proof}
Inserting $f= \sum_{k\in \zn} \delta_k$ in \eqref{3.25} then it follows from \eqref{2.28} that
\begin{\eq}   \label{3.32}
F_\tau f = (2 \pi)^{n/2} \Big( \tau (\xi) \sum_{k\in \zn} \delta_k \Big)^\wedge = \sum_{k\in \zn} \tau (k) e^{-ixk}.
\end{\eq}
Similarly one has for $f= \sum_{k\in \zn} e^{i 2 \pi k \xi}$ that
\begin{\eq}   \label{3.33}
F_\tau f = (2\pi )^{n/2} \Big( \tau (\xi)  \sum_{k\in \zn} e^{i 2 \pi k \xi} \Big)^\wedge = (2\pi )^{n/2} \sum_{k\in \zn} \wh{\tau}
(x- 2\pi k).
\end{\eq}
Then  \eqref{3.31a} follows from \eqref{2.29}.
\end{proof}

\begin{remark}   \label{R3.7}
As indicated the identity \eqref{3.31a} must be understood in the context of the dual pairing $\big( S(\rn), S'(\rn) \big)$. One may ask for which 
$\tau$ the identity holds pointwise or belongs to suitable function spaces, for example $\Cc^s (\rn) = B^s_{\infty, \infty} (\rn)$,
$s >0$. But this will not be done here. Of course the above arguments can also be applied to symbols $\tau (\xi) \in \Phi^\sigma_{1,0}
(\rn)$, $\sigma \in \real$, being independent of $x\in \rn$. If the symbol $\tau = \tau (x, \xi) \in \Phi^\sigma_{1,\delta} (\rn)$,
$\sigma \in \real$, $0 \le \delta < 1$, depends on $x\in \rn$, then one has
\begin{\eq}   \label{3.34}
F_\tau f = \sum_{k\in \zn} \tau(x, k) \, e^{-ixk}
\end{\eq}
instead of \eqref{3.32} and
\begin{\eq} \label{3.35}
F_\tau f = \sum_{k\in \zn} \int_{\rn} e^{-i(x- 2\pi k)\xi} \tau(x, \xi) \, \di \xi
\end{\eq}
as the substitute of \eqref{3.33} resulting in the suitably interpreted identity
\begin{\eq}   \label{3.36}
\sum_{k\in \zn} \tau(x, k) \, e^{-ixk} = \sum_{k\in \zn} \int_{\rn} e^{-i(x- 2\pi k)\xi} \tau(x, \xi) \, \di \xi
\end{\eq}
in the framework of the dual pairing $\big( S(\rn), S'(\rn) \big)$. We do not know whether this (known or unknown) modified Poisson
summation formula is of any use. 
\end{remark}

\subsection{Diffeomorphisms}   \label{S3.5}
So far we inserted the identity \eqref{2.29} into distinguished  linear operators and obtained in this way diverse modifications of the
classical Poisson summation formula \eqref{1.1}. The identity \eqref{3.31a} may serve as an example. But one can try to extend first
the identity \eqref{2.29} and to proceed afterwards in the indicated way. Some modifications are quite straightforward. Let $A =
(a_{j,k})^n_{j,k =1}$, $a_{j,k} \in \real$, be an invertible matrix, generating an affine mapping $y=Ax$, $x\in \rn$, of $\rn$ onto 
itself. Let $L = A \zn$ be the related lattice, called Dirac comb. Then one can modify \eqref{2.29} replacing there
\begin{\eq}   \label{3.37}
\sum_{k\in \zn} \delta_k \quad \text{by} \quad \mu= \sum^N_{j=1} \sum_{\lambda \in L+h_j} \sum_{|m| \le M} c^j_m e^{im\lambda} 
\delta_\lambda,
\end{\eq}
$h_j \in \rn$, $c^j_m \in \comp$, and a correspondingly adapted left--hand side. It is the main aim of \cite{LeO15, LeO17} to prove
that a positive measure $\mu$ in $\rn$ with uniformly discrete support such that $\wh{\mu}$ has also a discrete closed support 
(spectrum) must be necessarily as described in \eqref{3.37} for suitably chosen $L, N, h_j$ and $c^j_m$. But the study of the atomic
structure of crystals and so--called quasi--crystals suggests to ask for modified Poisson summation formulas with underlying sets of
off--points which are less regular than indicated in \eqref{3.37}. One may consult the references \cite{LeO15, LeO17, Mey16, Mey17}
and in particular \cite{Gon23, GoV25} already mentioned in the Introduction \ref{S1}. We indicate how one can contribute to this topic
in the framework of the theory of function spaces. Of peculiar interest might be isomorphic mappings of the spaces $\As (\rn)$, $A \in
\{B,F \}$, including  $\Cc^{-n} (\rn) = B^{-n}_{\infty, \infty} (\rn) = F^{-n}_{\infty, \infty} (\rn)$, onto itself based on
diffeomorphic mappings of $\rn$ onto itself. We give a brief description following \cite[Section 2.3, pp.\,39--40]{T20}. There one 
finds further explanations and references.

A continuous  one--to--one mapping of $\rn$, $n\in \nat$,  onto itself,
\begin{align}  \label{3.38}
y &= \psi (x) = \big(\psi_1 (x), \ldots, \psi_n (x) \big), && \text{$x\in \rn$}, &&  \\  \label{3.39}
x &= \psi^{-1} (y) = \big( \psi^{-1}_1 (y), \ldots, \psi^{-1}_n (y) \big), && \text{$y\in \rn$}, &&
\end{align}
is called a {\em  diffeomorphism} if all components $\psi_j (x)$ and of its inverse $\psi^{-1}_j (y)$ are real $C^\infty$ functions on $\rn$ and for $j=1, \ldots,n$,
\begin{\eq}   \label{3.40}
\sup_{x\in \rn} \big( | D^\alpha \psi_j (x) | + | D^\alpha \psi^{-1}_j (x)| \big) < \infty \quad  \text{for all $\alpha \in \nat^n_0$
with $|\alpha| >0$}.
\end{\eq}
Then $\vp \mapsto \vp \circ \psi$, given by $(\vp \circ \psi )(x) = \vp \big( \psi(x) \big)$, is an one--to--one mapping of $S(\rn)$ onto itself. This can be extended to an one--to--one mapping of $S'(\rn)$ onto itself by
\begin{\eq}   \label{3.41}
(f \circ \psi )(\vp) = \big( f, |\mathrm{det\,}\psi^{-1}_* | \, \vp \circ \psi^{-1} \big), \qquad f \in S' (\rn), \quad \vp \in S(\rn),
\end{\eq}
as the distributional version of
\begin{\eq}  \label{3.42}
\int_{\rn} \big( f \circ \psi \big) (x) \, \vp (x) \, \di x = \int_{\rn} f(y) \, \big(\vp \circ \psi^{-1} \big)(y) \big|\mathrm{det \,}
\psi^{-1}_* \big|(y) \, \di y,
\end{\eq}                                                                                                                            
(change of variables) where $\psi^{-1}_*$ is the Jacobian and $\mathrm{det \, } \psi^{-1}_*$ its determinant. Recall that $\big|\mathrm{det \,}\psi^{-1}_* \big|(y)$ is a $C^\infty$ function in $\rn$ and
\begin{\eq} \label{3.43}
c_1 \le \big|\mathrm{det \,} \psi^{-1}_* \big|(y) \le c_2 \quad \text{for all $y\in \rn$ and some $0<c_1 \le c_2 <\infty$}.
\end{\eq}
Then
\begin{\eq}   \label{3.44}
D_\psi: \quad \As (\rn) \hra \As (\rn), \qquad D_\psi f = f \circ \psi
\end{\eq}
is an isomorphic mapping for all spaces $\As (\rn)$ with $s\in \real$ and $0<p,q \le \infty$. 

We apply the above considerations to \eqref{2.29}.

\begin{proposition}   \label{P3.8}
Let $\psi$ be the above diffeomorphic mapping of $\rn$ onto itself. Then
\begin{\eq}   \label{3.45}
\begin{aligned}
\sum_{k \in \zn} e^{i2\pi k\psi (x)} &= \sum_{k\in \zn} \big|\mathrm{det \,} \psi^{-1}_* \big|(k) \, \delta_{\psi^{-1} (k)} \\
&= \sum_{k \in \zn} \big|\mathrm{det \,} \psi_* \big|^{-1} \delta_{\psi^{-1} (k)} \in \Cc^{-n} (\rn),
\end{aligned}
\end{\eq}
convergence being in $S'(\rn)$.
\end{proposition}
 
\begin{proof}
One has by \eqref{3.41} with $f= \delta_z$, $z\in \rn$, and well--known properties of Jacobian matrices  that
\begin{\eq}    \label{3.46}
\begin{aligned}
(\delta_z \circ \psi)(\vp) &= \big( \delta_z, \big|\mathrm{det \,} \psi_*^{-1} \big| \vp \circ \psi^{-1} \big) \\
&=  \big|\mathrm{det \,} \psi_*^{-1} \big| (z) \, \vp \big( \psi^{-1} (z) \big) \\
&=  \big|\mathrm{det \,} \psi_*^{-1} \big| (z) \, \delta_{\psi^{-1} (z)} (\vp) \\
&=  \big|\mathrm{det \,} \psi_* \big|^{-1} \big( \psi^{-1} (z) \big) \delta_{\psi^{-1} (z)} (\vp), \quad \vp\in S(\rn).
\end{aligned}
\end{\eq}
Then \eqref{3.45} follows from \eqref{3.44} applied to \eqref{2.29}.   
\end{proof}

The coefficients $\big|\mathrm{det \,} \psi_* \big|^{-1} \big( \psi^{-1} (k) \big)$ on the right--hand side of \eqref{3.45} are all
positive. One can remove this shortcoming (if it is any) at least partly by finite linear combinations of related identities. But more
interesting might be the application of a further significant property of the theory of function spaces, saying that $\Cc^{\vr} (\rn)$,
$\vr >n$, is a pointwise multiplier for the space $\Cc^{-n} (\rn)$, which means that there is a constant $c>0$ such that
\begin{\eq}   \label{3.47}
\| gf \, | \Cc^{-n} (\rn) \| \le c \, \|g \, |\Cc^{\vr} (\rn) \| \cdot \| f \, | \Cc^{-n} (\rn) \|
\end{\eq}
for all $g \in \Cc^{\vr} (\rn)$ and all $f\in \Cc^{-n} (\rn)$. This is a special case of \cite[Theorem 2.30, p.\,41]{T20}. There one
finds further explanations and references. Now one can combine this observation with Proposition \ref{P3.8}.

\begin{theorem}   \label{T3.9}
Let $\psi$ be the above diffeomorphic mapping of $\rn$ onto itself and let $g\in \Cc^{\vr} (\rn)$ with $\vr >n\in \nat$. Then
\begin{\eq}   \label{3.48}
g(x) \sum_{k \in \zn} e^{i2\pi k\psi (x)} 
= \sum_{k \in \zn} \frac{g}{\big|\mathrm{det \,} \psi_* \big|} \delta_{\psi^{-1} (k)} \in \Cc^{-n} (\rn),
\end{\eq}
convergence being in $S'(\rn)$.
\end{theorem}

\begin{proof}
This is the immediate consequence of Proposition \ref{P3.8} and \eqref{3.47}.
\end{proof}

\begin{remark}   \label{R3.10}
The above considerations, the quoted papers and the literature within suggest to have a closer look at the one--dimensional case $n=1$.
Let $y= \psi (x)$, $x \in \real$, be a smooth monotonically increasing function with $0<c_1 \le \psi' (x) \le c_2 <\infty$ and let $g\in
\Cc^{\vr} (\real)$ with $\vr >1$. Then it follows from \eqref{3.48} that
\begin{\eq}   \label{3.49}
g(x) \sum_{k \in \ganz} e^{i2 \pi k \psi(x)} = \sum_{k\in \ganz} \frac{g}{\psi'} \, \delta_{\psi^{-1} (k)}.
\end{\eq}
\end{remark}

Now one can proceed in $\rn$ as above replacing \eqref{2.29} by \eqref{3.48} with the following outcome.

\begin{proposition}    \label{P3.11}
Let $\tau \in \Phi (\rn)$ according to \eqref{3.31}. Let $\psi$ be the above isomorphic mapping of $\rn$ onto itself and let $g\in
\Cc^{\vr} (\rn)$ with $\vr >n$. Then
\begin{\eq}  \label{3.50}
\sum_{k \in \zn} \Big( \frac{\tau g}{\big|\mathrm{det \,} \psi_* \big|}\Big) \big(\psi^{-1} (k)\big) e^{-ix \psi^{-1}(k)} =
(2\pi)^{n/2} \sum_{k\in \zn} \Big( (\tau g)(\cdot) e^{i2\pi k\psi(\cdot)} \Big)^\wedge,
\end{\eq}
convergence being in $S'(\rn)$.
\end{proposition}

\begin{proof}
One can argue as in the proof of Proposition \ref{P3.6} inserting separately both sides of \eqref{3.48} into $F_\tau$ according to
\eqref{3.25} with $\tau \in \Phi (\rn)$. Then \eqref{3.50} follows from
\begin{\eq}   \label{3.51}
F_\tau f = \sum_{k \in \zn} \Big( \frac{\tau g}{\big|\mathrm{det \,} \psi_* \big|}\Big) \big(\psi^{-1} (k)\big) e^{-ix \psi^{-1}(k)}
\end{\eq}
for $f= \sum_{k\in \zn} \frac{g}{|\mathrm{det \,} \psi_* |} \delta_{\psi{-1} (k)}$ and
\begin{\eq}  \label{3.52}
F_\tau f = (2\pi)^{n/2} \sum_{k\in \zn} \Big( (\tau g)(\cdot) e^{i2\pi k\psi(\cdot)} \Big)^\wedge,
\end{\eq}
for $f= g(x) \sum_{k \in \zn} e^{i2 \pi k \psi(x)}$.
\end{proof}

\begin{remark}   \label{R3.12}
For $\psi(\xi) = \xi$ and $g=1$ in \eqref{3.50} one obtains \eqref{3.31a}. One has now the Fourier transform of the modulated function
$\tau g$ instead of $\wh{\tau} (x- 2\pi k)$. This simplifies if $\psi$ is the affine mapping as discussed in connection with 
\eqref{3.37}.
\end{remark}

\section{Complements}   \label{S4}
\subsection{Poisson summation formula, limiting assertions}   \label{S4.1}
For $\vp \in S(\rn)$ the convergence of the classical Poisson summation formula \eqref{1.1} and of its extended version in 
\eqref{2.19} are obvious. Greater care is needed in connection with \eqref{1.3} and Theorem \ref{T2.1}. Convergence must now be 
understood in the framework of the dual pairing  $\big(S(\rn), S'(\rn) \big)$. Similarly in the applications as described above. One
may always ask afterwards whether the related series converge already in some distinguished function spaces or even pointwise. In
some cases this is quite obvious, but we did not stress this point above. On the other hand it is of interest to extend \eqref{1.1}
and \eqref{2.19} to further functions $f$ in place of $\vp \in S(\rn)$ and to clarify how the related convergence must be understood.
This has been done in the classical context in \cite[Theorem 3.1.17, p.~167]{Gra04} assuming that both $f$ and $\wh{f}$ belong to
$L_1 (\rn)$ and
\begin{\eq}   \label{4.0}
|f(x)| + |\wh{f} (x)| \le c (1+|x|)^{-n-\delta}, \qquad x\in \rn,
\end{\eq}
for some $c>0$ and $\delta >0$. This covers in particular the arguments in Section \ref{S3.3} for the Cauchy--Poisson semi--group.
We ask for natural function spaces such that \eqref{1.1} and \eqref{2.19} make sense for its elements. For this purpose we recall the
duality
\begin{\eq}   \label{4.1}
B^n_1 (\rn)' = \Cc^{-n} (\rn)
\end{\eq}
according to \cite[Theorem 2.11.2, p.~178]{T83} where the Besov space $B^n_1 (\rn) = B^n_{1,1} (\rn)$ can be characterize as the
collection of all $f\in L_1(\rn)$ such that for $M \in \nat$ with $M >n$,
\begin{\eq}   \label{4.2}
\| f \, |B^n_1 (\rn) \|_M = \|f | L_1 (\rn)\| + \int_{x\in \rn, h\in \rn, |h| \le 1} \big| \Delta^M_h f(x) \big| 
\frac{\di x \, \di h}{|h|^{2n}} <\infty,
\end{\eq}
(equivalent norms), where $\Delta^m_h f$ are the usual differences in \rn, \cite[Corollary 2.6.1/1, p.~142]{T92}. One has for the dual
pairing of $g\in \Cc^{-n} (\rn)$ and $f\in B^n_1 (\rn)$ that
\begin{\eq}   \label{4.3}
|(g,f)| \le \|g \, |\Cc^{-n} (\rn) \| \cdot \|f \, | B^n_1 (\rn) \|.
\end{\eq}
Using that $S(\rn)$ is dense in $B^n_1 (\rn)$ it follows from \eqref{4.3} applied to $g\in \Cc^{-n}(\rn)$ and $f_j \to f$ in $B^n_1
(\rn)$ for any approximating sequence $\{f_j \} \subset S(\rn)$ that
\begin{\eq}   \label{4.4}
(g, f_j ) \to (g,f) \quad \text{if $j \to \infty$ in $\comp$}.
\end{\eq}
Recall that
\begin{\eq}  \label{4.5}
B^n_1 (\rn) \hra C(\rn)
\end{\eq}
where $C(\rn)$ is again the space of all complex bounded continuous functions in \rn, \cite[Theorem 2.3, p.~22]{T20} and the references
given there. Then it follows from \eqref{4.4} with $g = \sum_{k\in \zn} \delta_k \in \Cc^{-n} (\rn)$ that
\begin{\eq}    \label{4.6}
\Big( \sum_{k \in \zn} \delta_k, f \Big) = \lim_{j \to \infty} \Big(\sum_{k\in \zn} \delta_k, f_j \Big) = \lim_{j \to \infty}
\sum_{k\in \zn} f_j (k).
\end{\eq}
The embedding \eqref{4.5} justifies to re--write  the left--hand side of \eqref{4.6} as $\sum_{k\in \zn} f(k)$. Similarly for $g =
\sum_{k\in \zn} e^{i2\pi k \xi} \in \Cc^{-n} (\rn)$. We fix the outcome, extending \eqref{2.19} from $\vp \in S(\rn)$ to $f\in B^n_1
(\rn)$  where convergence must always be understood as just explained.

\begin{proposition}   \label{P4.1}
Let $f\in B^n_1 (\rn)$. Then
\begin{\eq}   \label{4.7}
\sum_{k\in \zn} f(x+ 2\pi k) = (2 \pi)^{-n/2} \sum_{k\in \zn} \wh{f}(k) \, e^{ikx}, \qquad x\in \rn.
\end{\eq}
\end{proposition}

\begin{proof} The identity \eqref{4.7} follows for $x=0$ from \eqref{1.1} and the above explanations with $\delta_{2\pi k}$ in place of
$\delta_k$ and $e^{ik \xi}$ in place of $e^{i2\pi k \xi}$. Replacing there $f(\cdot)$ by
$f(x+ \cdot)$, $x\in \rn$, one obtains \eqref{4.7}.
\end{proof}

\begin{remark}   \label{R4.2}
One can use \eqref{4.7} to justify Proposition \ref{P3.4} more directly. Let
\begin{\eq}  \label{4.8}
f(x) = (2\pi)^{n/2} c_n \frac{t}{(|x|^2 + t^2)^{\frac{n+1}{2}}}, \qquad t>0, \quad x\in \rn,
\end{\eq}
be the right--hand side of\eqref{3.18}. Then $f\in B^n_1 (\rn)$ follows from $D^\alpha f \in L_1(\rn)$, $\alpha \in \nat^n_0$, and
elementary embeddings. One has by \eqref{3.18} that $\wh{f} (\xi) = e^{-t |\xi|}$. Inserted in \eqref{4.7} with $x=0$ one obtains
\begin{\eq}   \label{4.9}
(2\pi)^n c_n  t \sum_{k\in \zn} \frac{1}{(|2\pi k|^2) + t^2)^{\frac{n+1}{2}}} = \sum_{k\in \zn} e^{-t |k|}.
\end{\eq}
Then \eqref{3.19} follows for $x=0$ from \eqref{4.9} with $2\pi t$ in place of $t$. The extension from $x=0$ to $x\in \rn$ can be
done as in the proof of Proposition \ref{P4.1}.
\end{remark}

\subsection{Some relations}   \label{S4.2}
Our arguments rely on limiting assertions, both for specific  embeddings and also for mappings of the Fourier transform. They are part
of a larger landscape which might be fixed also for further investigations. In particular it shows that Proposition \ref{P4.1} is
also a natural limiting assertions. 

\subsubsection{Some limiting embeddings}  \label{S4.2.1}
Let $C(\rn)$ be again  the space of all complex bounded continuous functions in \rn. Let $0<q \le \infty$. Then
\begin{\eq}   \label{4.10}
\id: \quad B^0_{\infty,q} (\rn) \hra L_\infty (\rn) \qquad \text{if, and only if, $0<q \le 1$},
\end{\eq}
\begin{\eq}    \label{4.11}
\id: \quad B^0_{\infty,q} (\rn) \hra  C(\rn) \qquad \text{if, and only if, $0<q \le 1$},
\end{\eq}
is covered by \cite[Theorem 2.3, p.~22]{T20}, whereas one has by \cite[Proposition 2.43, p.~46]{T20} that
\begin{\eq}    \label{4.12}
\id: \quad L_\infty (\rn)\hra B^0_{\infty,q} (\rn) \qquad \text{if, and only if, $q=\infty$}.
\end{\eq}
This can be complemented by
\begin{\eq}    \label{4.13}
\id: \quad C(\rn) \hra B^0_{\infty,q} (\rn) \qquad \text{if, and only if, $q=\infty$}
\end{\eq}
as already mentioned in \cite[Remark 3.11, p.~112]{SiT95}. But it can be proved rather quickly as follows. Let $C(\rn)$ be continuously
embedded in $B^0_{\infty,q} (\rn)$ for some $q$, $0<q \le \infty$. The characteristic function $\chi_Q$ of a cube $Q$ in $\rn$ can be
approximated in, say, $L_2 (\rn)$, and hence in $S'(\rn)$, by suitable continuous functions. Then it follows from the Fatou property of
$B^0_{\infty,q} (\rn)$ according to \cite[Theorem 1.25, p.~18]{T20} (and the references given there) that $\chi_Q \in B^0_{\infty,q}
(\rn)$. But this requires by \cite[Proposition 2.43, p.~46]{T20} that $q=\infty$. These assertions can be complemented by
\begin{\eq}   \label{4.14}
B^0_{1,u} (\rn) \hra L_1 (\rn) \hra B^0_{1,v} (\rn) \ \text{if, and only if, both $0<u \le 1$ and $v=\infty$}.
\end{\eq}
The if--part follows from the Fourier--analytic definition of $B^0_{1,q} (\rn)$. The only --if--part can be obtained by duality from
\eqref{4.10} and \eqref{4.12}. These properties are more or less known and already covered by \cite{SiT95} and may also be found in
\cite[Section 2.3.3, pp.~43--44]{ET96}. This applies also to the embedding
\begin{\eq}   \label{4l.15}
\id: \quad B^n_1 (\rn) \hra B^0_{\infty,1} (\rn) \hra C(\rn),
\end{\eq}
\cite[Theorem 2.5, p.~25]{T20}.

\subsubsection{Some mappings}   \label{S4.2.2}
We illuminate the above considerations  by collecting some mapping properties of the Fourier transform $F$ without further discussions.
First we mention that
\begin{\eq}   \label{4.15}
F: \quad B^{n+s_1}_1(\rn) \hra B^{s_2}_1 (\rn) \quad \text{if, and only if, both $s_1 \ge 0$ and $s_2 \le 0$}.
\end{\eq}
This is a special case of the related assertion  in \cite{HST23}. Then it follows from \eqref{4.14} that
\begin{\eq}   \label{4.16}
F: \quad B^n_1 (\rn) \os{F}{\hra} B^0_1 (\rn) \os{\id}{\hra} L_1 (\rn)
\end{\eq}
and by duality according to \cite[Theorem 2.11.2, p.~178]{T83} that
\begin{\eq}   \label{4.17}
F: \quad \Cc^0 (\rn) \hra \Cc^{-n}(\rn).
\end{\eq}
This is one of the cornerstones in \cite{HST23}. The mappings \eqref{4.15} and \eqref{4.16} complement \eqref{4.7} and the discussion
how the related convergence must be understood. Finally we mention the special case
\begin{\eq}    \label{4.18}
F: \quad B^n_1 (\rn) \hra \Cc^0 (\rn, w_n)
\end{\eq}
according to \cite[Theorem 2.1]{Tri23}, where $w_n$ is the same weight as in \eqref{3.1} (with $\alpha =n$) and $\Cc^0 (\rn, w_n)$ is
normed by $\| w_n f | \Cc^0 (\rn) \|$. Since  $S(\rn)$ is dense in $B^n_1 (\rn)$ one can improve \eqref{4.18} by
\begin{\eq}   \label{4.19}
F: \quad B^n_1 (\rn) \hra \os{\circ}{\Cc}{}^0 (\rn, w_n),
\end{\eq}
the completion of $S(\rn)$ in $\Cc^0(\rn, w_n)$. Then it follows from the duality in \cite[Theorem 2.11.2, Remark 2.11.2/2, 
pp.~178/180]{T83} that also
\begin{\eq}   \label{4.20}
F: \quad B^0_1 (\rn, w_{-n}) \hra \Cc^{-n} (\rn)
\end{\eq}
with $B^0_1 (\rn, w_{-n} )$ normed by $\|w_{-n} f \, | B^0_1 (\rn) \|$.

\end{document}